\documentclass[10pt,reqno]{amsart}
\usepackage{amsmath,amsfonts,amsthm,amssymb,amsxtra}
\usepackage{amsxtra, amssymb, mathrsfs,dsfont}

\newtheorem{theorem}{Theorem}[section]
\newtheorem{proposition}[theorem]{Proposition}

\newtheorem{lemma}[theorem]{Lemma}
\newtheorem{corollary}[theorem]{Corollary}

\theoremstyle{definition}

\newtheorem{assumption}[theorem]{Assumption}

\theoremstyle{remark}
\newtheorem{remark}[theorem]{\bf Remark}
\numberwithin{equation}{section}

\newcommand{\A}{\mathcal{A}}
\newcommand{\B}{\mathfrak{B}}

\newcommand{\Q}{\mathcal{Q}}
\newcommand{\V}{\mathcal{V}}
\newcommand{\W}{\mathcal{W}}
\newcommand{\U}{\mathcal{U}}

\newcommand{\pd}{\partial}
\newcommand{\eps}{\varepsilon}

\newcommand{\id}{\mathds{1}}

\newcommand{\N}{\mathbb{N}}

\newcommand{\R}{\mathbb{R}}
\newcommand{\Sph}{\mathbb{S}}

\newcommand{\Z}{\mathbb{Z}}

\begin{document}

\title[Two-dimensional magnetic Schr\"odinger operators]
{Eigenvalue bounds for two-dimensional magnetic
Schr\"odinger operators}

\author {Hynek Kova\v{r}\'{\i}k}

\address {Hynek Kova\v{r}\'{\i}k, Dipartimento di Matematica, Politecnico di Torino}

\email {Hynek.Kovarik@polito.it}

%%%%%%%%%%%%%%%%%%%%%%%%%%%%%%%%%%%%%%

\date{\today}

\begin {abstract}
We prove that the number of negative eigenvalues of two-dimensional magnetic Schr\"o-dinger operators is bounded from above by the strength of the corresponding electric potential.  Such estimates fail in the absence of a magnetic field.
We also show how the corresponding upper bounds depend on the properties of the magnetic field and discuss their connection with Hardy-type inequalities.  
\end{abstract}

\maketitle



%%%%%%%%%%%%%%%%%%%%%%%%%%%%%%%%%%%%%%%%%%%%%%%%%%%%%%%%%%%%%%%%%%%%%%%%%%%

\section{\bf Introduction}

The Hamiltonian of a charged quantum particle in $\R^2$ interacting with
a magnetic field $B = \mbox{curl}\, A$ is given formally by the differential operator
\begin{equation} \label{ham-general}
H_B = ( i \nabla +A)^2 \qquad \text{in\, \, } L^2(\R^2).
\end{equation}
We will deal with spectral estimates for Schr\"odinger operators $H_B-V$, where $V$ is an additional electric potential. The well-known Cwikel-Lieb-Rosenblum inequality \cite{cw,lieb,ros} says that in dimension $d\geq 3$ the number $N(H_0-V,0)$ of negative eigenvalues of $H_0-V$ can be estimates as follows;
\begin{equation} \label{clr-classic}
N(H_0-V,0)=N(-\Delta-V,0) \, \leq \, C_d \int_{\R^d} V_+(x)^{d/2}\, dx, \qquad d\geq 3,
\end{equation}
where $V_+$ denotes the positive part of $V$ and $C_d$ is a constant independent of $V$. Moreover, in \cite{ahs} it is shown that inequality \eqref{clr-classic} holds, under certain generic assumptions,  with the same constant $C_d$ also in the presence of a magnetic field, i.e. with $-\Delta$ replaced by $H_B$. 

On the other hand, it is also known that \eqref{clr-classic} fails if $d=2$. This is clear already from the fact that the operator $-\Delta-V$ in dimension two has weakly coupled eigenvalues, in other words if $\int_{\R^2} V \geq 0$, then $N(-\Delta-\lambda V,0) \geq 1$ for any $\lambda >0$. There are also other, less obvious, reasons behind the failure of \eqref{clr-classic} for $d=2$, see section \ref{sect:discussion} for more details. 

When an additional magnetic field is introduced, then it is natural to expect that the situation described above might improve (in certain sense) due to diamagnetic effects. Indeed, it is known that  
magnetic Schr\"odinger operators typically do not have weakly coupled eigenvalues, \cite{timo}. Therefore 
we address the question whether it is possible to establish an analogue of the CLR-inequality \eqref{clr-classic} in dimension two for the counting function $N(H_B-V,0)$. This problem was solved in \cite{bel} in the case of the Aharonov-Bohm magnetic field represented by a Dirac delta function, see Remark \ref{rem:ab} below.

\smallskip

However, it is easily seen that as soon as a (radial) magnetic field is not of the Aharonov-Bohm type,  in other words when it is more regular, then the estimate proved in \cite{bel}, namely inequality \eqref{eq:bel} below,  must fail, see Proposition \ref{examples}. Our aim is thus to establish a suitable upper bound on $N(H_B-V,0)$ for a reasonably large class of magnetic fields and in particular to find out how such an upper bound depends on the properties of $B$. 

In the first part of the paper we prove a weighted version of \eqref{clr-classic} for general magnetic fields, 
see Theorems \ref{clr-mag-1} and \ref{clr-mag-2}.  The proofs of these theorems are based on a modification of the method of Lieb \cite{lieb,ahs,rs} and on certain Hardy type inequalities for the operator $H_B$ obtained in \cite{lw,timo}. The advantage of such approach is that is can be applied to a very large class of magnetic fields. Moreover, it also enables us to prove a family of weighted Sobolev inequalities for the operator $H_B$, see Corollary \ref{cor-sobolev}, which might be of independent interest. 
On the other hand, the upper bounds obtained by this method do not have the correct behavior in the strong coupling regime, cf. Remark \ref{rem:semiclassics}. 

Therefore, in the second part of the paper, we show that for radial magnetic fields with finite total flux one can establish sharper estimates on $N(H_B-V,0)$ with the expected strong coupling behavior, see Theorems \ref{clr-radial} and \ref{clr-radial-integer}. It is interesting to notice that the integral weights involved in these bounds change according to the value of the total flux of the magnetic field.  It turns out that this phenomenon is directly related to the decay rate of the weight functions of the respective 
Hardy-type inequalities for the operator $H_B$, see section \ref{sec:hardy-integer} for further details.

%%%%%%%%%%%%%%%%%%%%%%%%%%%%%%%%%%%%%%%%%%%%%%%%%%%%%%%%%%%%%%%%%%%%%%%%%%%%
\section{\bf Preliminaries and notation}

\noindent Given a self-adjoint operator $T$ on a Hilbert space $\mathcal{H}$, we denote by $N(T,\,  s)_{\mathcal{H}}$ the number of its discrete eigenvalues (counted with multiplicities) below $s\in\R$.
If $\mathcal{H}= L^2(\R^2)$, then we omit the subscript and write $N(T,\, s)$. For two functions $f_1,\, f_2$ on a set $\Omega$ we will
use the notation 
\begin{equation} \label{def:as}
f_1(x)  \, \simeq \,  f_2(x) \quad  \Leftrightarrow \quad  \exists\,  \, c>0 :\, \,  c^{-1}\, f_1(x) \, \leq \, 
f_2(x)\,  \leq \,  c\, f_1(x) \qquad  \forall \ x\in\Omega.
\end{equation}
An important characteristics of the magnetic field is its flux $\Phi(r)$ through the disc of radius $r$ centered in the origin:
\begin{equation} \label{flux}
\Phi(r) = \frac{1}{2\pi} \int_{\{x:\, |x|\leq r\}} B(x)\, dx. 
\end{equation}
We will denote the total flux of $B$ by 
$$
\Phi = \frac{1}{2\pi} \int_{\R^2} B(x)\, dx
$$ 
whenever the above integral is finite. Finally, we will use the notation $(\cdot\, , \cdot)_{\mathcal{H}}$ for the scalar product in a Hilbert space $\mathcal{H}$, and $(\cdot\, , \cdot)$ in the case $\mathcal{H}= L^2(\R^2)$.

%%%%%%%%%%%%%%%%%%%%%%%%%%%%%%%%%%%%%%%%%%%%%%%%%%%%%%%%%%%%%%%%%%%%%%%%%%%%%%%%%%%%%%%%%%%%%%%%%%

\section{\bf Main Results}

Since we are interested only in upper bounds on $N(H_B-V, 0)$, we may suppose without loss of generality that $V$ is non-negative. Moreover, we will always assume that $A\in L^2_{loc}(\R^2)$ and $V\in L^1_{loc}(\R^2)$. Under the symbol $H_B-V$ we will understand the Friedrichs extension of the operator generated by the quadratic form
\begin{equation} \label{quad-form}
\int_{\R^2} \big( |(i\, \nabla + A)\, u|^2- V\, |u|^2\big)\, dx, \qquad u\in C_0^\infty(\R^2),
\end{equation}
provided this form is bounded from below.

%%%%%%%%%%%%%%%%%%%%%%%%%%%%%

\smallskip

\subsection{Eigenvalue bounds for general magnetic fields}  

\begin{theorem} \label{clr-mag-1}
Assume that $A\in L^2_{loc}(\R^2)$ generates a non-zero magnetic field $B$. Let $0\leq V\in L^1_{loc}(\R^2)$ be such that the right hand side of \eqref{eq-clr-1} is finite for some $a>0$. Then the quadratic form \eqref{quad-form} is closable and there exists a constant $C=C(B,a)$, independent of $V$, such that 
\begin{equation} \label{eq-clr-1}
N(H_B-V, 0) \, \leq \, C\, \Big( \int_{\R^2} V(x)\, (1+|\log |x||)^{1+a} \, dx + \int_{\R^2} V(x)\, \log(1+V(x)) \, dx \Big).
\end{equation}
\end{theorem}

\noindent For the next result we will need more hypotheses on the magnetic field. The following condition is taken from \cite{lw}.
\begin{assumption} \label{ass-field}
Assume that there exist  $\eps \in(0,1/2), A=A(\eps)$ and a finite or infinite number of open intervals $I_j=(\alpha_j, \beta_j)$, such that 
\begin{align*}
\{ r>0\, :\, \min_{k\in\Z} |k-\Phi(r)| < \eps\} & \subset \cup_{j=1}^N\, I_j ,\\
\beta_{j-1} & < \alpha_j < \beta_j, \quad j=1,\dots, N, \\
|I_j| & \leq A\, \min_{1\leq j\leq N}\, \{1+\alpha_j, \alpha_j-\beta_{j-1}, \alpha_{j+1} -\beta_j\}.
\end{align*} 
\end{assumption}

\begin{remark} \label{rem1}
Assumption \ref{ass-field} requires that the flux $\Phi(r)$ does not stabilize on integers in long intervals. It is satisfied, for example, if the total flux of the magnetic field is finite and non-integer.    
\end{remark}

\begin{theorem} \label{clr-mag-2}
Assume that $A\in L^2_{loc}(\R^2)$ generates a magnetic field $B$ which satisfies assumption \ref{ass-field}.  Let $0\leq V\in L^1_{loc}(\R^2)\cap L^{1+a}(\R^2, (1+|x|)^{2a}\, dx)$ for some $a>0$. Then the quadratic form \eqref{quad-form} is closable and there exists a constant $C(B,a)$, independent of $V$, such that 
\begin{equation} \label{eq-clr-2}
N(H_B-V, 0) \, \leq \, C(B,a) \int_{\R^2} V(x)^{1+a}\, (1+|x|)^{2a}\, dx.
\end{equation} 
\end{theorem}

\noindent As a consequence of Theorem \ref{clr-mag-2} we obtain 

\begin{corollary} \label{cor-sobolev}
Assume that $A\in L^2_{loc}(\R^2)$ generates a magnetic field $B$ which satisfies assumption \ref{ass-field}. Then for any $q\in [2,\infty)$ there exists a constant $S_q>0$ such that 
\begin{equation} \label{eq:sobolev}
\int_{\R^2}  |(i \nabla + A)\, u|^2\, dx \, \geq \, S_q \Big ( \int_{\R^2} |u(x)|^q\, (1+|x|)^{-2}\, dx\Big)^{\frac 2q}
\end{equation}
holds for all $u\in C_0^\infty(\R^2)$. In particular, if $|A| \in L^\infty(\R^2)$, then \eqref{eq:sobolev} holds for all $u\in H^1(\R^2)$.
\end{corollary}

\noindent  Inequality \eqref{eq:sobolev} fails, for any $q$, if the magnetic field is absent, cf.~Remark \ref{counter-ex}.

\smallskip

\begin{remark} ({\bf Semiclassical behavior}). \label{rem:semiclassics} 
Since a magnetic field does not affect the classical phase space volume, under certain generic decay conditions on $V$ the counting function $N(H_B+ \lambda\, V, 0)$ will obey the Weyl asymptotical formula
\begin{equation} \label{weyl}
\lim_{\lambda\to\infty} \lambda^{-1} N(H_B- \lambda\, V, 0) = \frac{1}{4\pi}\, \int_{\R^2} V(x)\, dx,
\end{equation}
see e.g.~\cite{sob}. On the other hand, introducing a coupling constant $\lambda$ in front of $V$ we easily see that when $\lambda\to\infty,$ then the right hand sides of  \eqref{eq-clr-1} and \eqref{eq-clr-2} are proportional to $\lambda\, \log\lambda$ and $\lambda^{1+a}$ respectively. In other words, they grow too fast with $\lambda$. This common defect of the bounds  \eqref{eq-clr-1} and \eqref{eq-clr-2} cannot be avoided within the approach used in their proofs. 

However, in the next section we will show that it can be removed, applying a different method, under the condition that the magnetic field is radial.
\end{remark}

%%%%%%%%%%%%%%%%%%%%%%%%%%%%%

\subsection{Eigenvalue bounds for radial magnetic fields}
\label{sect-radial}
For radial magnetic fields have stronger versions of Theorems \ref{clr-mag-1} and \ref{clr-mag-2} and . To state them we need some notation. We say that a potential function $V$ belongs to the class 
$L^1(\R_+,L^\infty(\Sph^1))$ if 
\begin{equation} \label{L1p}
\|V\|_{L^1(\R_+, L^\infty(\Sph^1))}  = \int_0^{\infty} \tilde V(r)\,  r\, dr \, < \infty \, ,
\end{equation}
where 
\begin{equation}
\tilde V(r) : = \text{ess} \sup_{0\leq \theta\leq 2\pi} |V(r,\theta)|.
\end{equation}
Moreover, given $s>0$ we denote $\B_s:=\{x\in \R^2\, :\, |x| < s \}$. 

\smallskip

\begin{assumption} \label{ass-field-radial}
Let $B\in  L^1(\R_+, (1+r)\, dr)$ be real-valued function and assume that $B(x)=B(|x|)$. 
\end{assumption}

\begin{theorem} \label{clr-radial}
Let $B$ satisfy Assumption \ref{ass-field-radial}.  Assume that $\Phi\notin\Z$. Suppose moreover that $V\in L^1_{loc}(\R^2, |\log |x||\, dx)$ and that
$V\in L^1(\R_+,L^\infty(\Sph^1))$. Then the quadratic form \eqref{quad-form} is closable and there exists a constant $C_1=C_1(B)$ , independent of $V$, such that
\begin{equation} \label{clr-eq-radial}
N(H_B-V, 0) \, \leq \, C_1\, \big( \|V \log |x| \|_{L^1(\B_1)}\, + \big \|V\big \|_{L^1(\R_+, L^\infty(\Sph^1))} \big) .
\end{equation}

\smallskip

\noindent In particular, if $V(x) =V(|x|)$, then 
\smallskip
\begin{equation} \label{clr-eq-radial-2}
N(H_B-V, 0) \, \leq \, C_1\, \big( \,  \|V\log |x| \|_{L^1(\B_1)} + \|V \|_{L^1(\R^2 )}\big).
\end{equation} 
\end{theorem}

\noindent If the total flux is an integer, then we have to replace the first term on the right hand side of \eqref{clr-eq-radial} by a corresponding $L^1-$norm of $V(x) \log (x)$ on the whole of $\R^2$:

\begin{theorem} \label{clr-radial-integer}
Let $B$ satisfy Assumption \ref{ass-field-radial}. Assume that $\Phi\in\Z$. Suppose moreover that $V\in L^1(\R^2, |\log |x||\, dx)$ and that
$V\in L^1(\R_+,L^\infty(\Sph^1))$. Then the quadratic form \eqref{quad-form} is closable
and there exists a constant $C_2=C_2(B)$, independent of $V$, such that
\begin{equation} \label{clr-eq-integer}
N(H_B-V, 0) \, \leq \, C_2 \big( \|V \log |x| \|_{L^1(\R^2)}\, +  \|V \|_{L^1(\R_+, L^\infty(\Sph^1))} \big).
\end{equation}

\smallskip

\noindent In particular, if $V(x) =V(|x|)$, then 
\smallskip
\begin{equation} \label{clr-eq-integer-2}
N(H_B-V, 0) \, \leq \, C_2 \big(  \|V\log |x| \|_{L^1(\R^2)} + \, \|V  \|_{L^1(\R^2 )} \big)
\end{equation} 
\end{theorem}

\smallskip

\noindent We note that, contrary to Theorems \ref{clr-mag-1} and \ref{clr-mag-2}, the upper bounds given in Theorems \ref{clr-radial} and \ref{clr-radial-integer} do respect the linear growth of $N(H_B+ \lambda\, V, 0)$ in $\lambda$ predicted by the Weyl formula \eqref{weyl}. Notice also that while in \eqref{clr-eq-radial} the logarithmic weight is only local, in \eqref{clr-eq-integer} it is included globally on the whole $\R^2$, which restricts the class of admissible potentials $V$. In the next section we will show that this restriction cannot be relaxed.

\begin{remark} ({\bf Aharonov-Bohm field}) \label{rem:ab} The vector potential 
\begin{equation} \label{ab-potential}
A(x) = \Phi\, \Big (-\frac{x_2}{|x|^2}\, ,\, \frac{x_1}{|x|^2} \Big) \qquad \text{on } \quad \R^2\setminus \{0\},
\end{equation}
generates the so-called Aharonov-Bohm magnetic field which corresponds to a Dirac delta placed in the origin. This field is fully characterized by its constant flux $\Phi=\Phi(r)$. The associated magnetic Hamiltonian, which we denote by $H_\Phi$, then satisfies 
\begin{equation} \label{eq:bel}
N(H_\Phi-V, 0) \, \leq \, C_\Phi  \,   \|V \|_{L^1(\R_+, L^\infty(\Sph^1))} ,
\end{equation}
where the constant $C_\Phi$ is finite if and only if $\Phi\notin\Z$. Estimate \eqref{eq:bel} was obtained in \cite{bel}. For the class of radial potentials $V$ a sharp value of the constant $C_\Phi$  was recently found by Laptev \cite{lap}. 
\end{remark}

%%%%%%%%%%%%%%%%%%%%%%%%%%%%%%%%%%%%%%%%%%%%%%%%

\section{\bf Discussion} \label{sect:discussion}

Inequalities \eqref{eq-clr-1}, \eqref{eq-clr-2} and \eqref{clr-eq-radial}, \eqref{clr-eq-integer} fail in the absence of magnetic field, since $N(H_0+\lambda\, V,0)\geq 1$ for all $\lambda>0$ provided $V$ is non-positive in the integral mean. In order to discuss the sharpness of the respective integral weights, we consider the following model potentials:
\begin{equation} \label{local}
V_\sigma(x) = \left\{
\begin{array}{l@{\quad}cr}
r^{-2}\, |\ln r|^{-2}\, |\ln |\ln r||^{-1/\sigma} & \text{if\, \,} & r<
e^{-2}  \\
0 & \, \, \text{if \, }  & r \geq e^{-2} 
\end{array}
\right. \qquad r=|x|,
\end{equation}
and  
\begin{equation} \label{global}
W_\sigma(x) = \left\{
\begin{array}{l@{\quad}cr}
r^{-2}\, |\ln r|^{-2}\, |\ln \ln r|^{-1/\sigma} & \text{if\, \,} & r>
e^{2}  \\
0 & \, \, \text{if \, }  & r \leq e^{2} 
\end{array}
\right. \qquad r=|x|,
\end{equation}
taken from \cite{bl}. Accordingly, we introduce the potential classes 
\begin{align}
\W_\sigma & := \Big\{\, 0<V\in L^1(\R^2)\, : \, V(x) =V(|x|),\ W_\sigma(x) = \mathcal{O}(V (x)), \quad |x| \to\infty \Big\},  \label{class-W} \\
\V_\sigma & := \Big\{\, 0<V\in L^1(\R^2)\, : \, V(x) =V(|x|),\ V_\sigma(x) = \mathcal{O}(V (x)), \qquad |x| \to 0 \Big\},  \label{class-V} 
\end{align}
which represent potentials with a slow decay at infinity and with a strong singularity in the origin, respectively.

One of the reasons for the failure of the Cwikel-Lieb-Rosenblum inequality in dimension two is the fact that for $\sigma>1$ the counting functions $N(-\Delta-\lambda\, V_\sigma,0)$ and $N(-\Delta-\lambda\, W_\sigma,0)$ have a super-linear growth in the coupling constant $\lambda$: 
\begin{equation} \label{birman-laptev}
N(-\Delta-\lambda\, V_\sigma,0) \, \sim\, N(-\Delta-\lambda\, W_\sigma,0) \, \sim \, \lambda^\sigma \qquad \text{as\, } \lambda\to\infty,
\end{equation} 
see \cite[Sec.6]{bl} for details. Below we show that this phenomenon occurs also for certain magnetic Schr\"odinger operators.

\begin{proposition} \label{examples}
Let $B(x)=B(|x|)$ be compactly supported and such that $B\in L^q(\R^2)$ for some $q>1$. Then 
\begin{equation} \label{noweyl:local}
\liminf_{\lambda\to\infty} \, \lambda^{-\sigma} N(H_B-\lambda\, V,0) \,  > \, 0 \qquad \forall\ V\in \V_\sigma,\ \sigma >1.
\end{equation}
If moreover $\Phi\in\Z$, then in addition to \eqref{noweyl:local} we also have 
\begin{equation} \label{noweyl:global}
\liminf_{\lambda\to\infty} \, \lambda^{-\sigma} N(H_B-\lambda\, V,0) \,  > \, 0 \qquad \forall\ V\in \W_\sigma,\ \sigma >1.
\end{equation}
\end{proposition}

\smallskip

\noindent 
Equation \eqref{noweyl:local} shows that estimate \eqref{eq:bel} must fail if the magnetic field satisfies conditions of Proposition \ref{examples}.

\smallskip

\begin{remark}
Proposition \ref{examples}, namely equation \eqref{noweyl:local}, shows that inequality \eqref{eq-clr-2} fails if $a=0$. Moreover, equation \eqref{noweyl:global} implies that Assumption \ref{ass-field} cannot be left out from Theorem \ref{clr-mag-2}. Indeed, since $W_\sigma\in L^1_{loc}(\R^2)\cap L^{1+a}(\R^2, (1+|x|)^{2a}\, dx)$ for any $a>0$,  for radial and compactly supported magnetic field with $\Phi\in\Z$ equation \eqref{noweyl:global} would be in contradiction with inequality \eqref{eq-clr-2}. As explained in Remark \ref{rem1}, such magnetic fields are excluded by Assumption \ref{ass-field}. 
\end{remark}	

\begin{remark}
Equation \eqref{noweyl:global} also tells us that the weight $(1+|\log |x||)^{1+a}$ in the first term on the rhs of \eqref{eq-clr-1} cannot be removed. Indeed, for a magnetic field with integer flux and $V=\lambda\, W_\sigma$ inequality \eqref{eq-clr-1} without the factor $(1+|\log |x||)^{1+a}$ would contradict equation \eqref{noweyl:global}. 
\end{remark}

\begin{remark}
The arguments of the previous remarks apply of course also to Theorems \ref{clr-radial} and \ref{clr-radial-integer}. Namely, equation \eqref{noweyl:local} shows that the logarithmic weight in the first term on the right hand side of \eqref{clr-eq-radial} cannot be omitted, while equation \eqref{noweyl:global} says that the condition $\Phi\not\in\Z$ in Theorem \ref{clr-radial} is necessary. In view of \eqref{noweyl:global}, the same reasoning implies that the term $\|V \log |x| \|_{L^1(\R^2)}$ on the right hand side of \eqref{clr-eq-integer} cannot be replaced by $\|V \log |x| \|_{L^1(\B_1)}$.
\end{remark}

%%%%%%%%%%%%%%%%%%%%%%%%%%%%%%%%%%%%%%%%%%%%%%%%%

\section{\bf Proofs of the main results: general fields}

We first prove the corresponding upper bounds on $N(H_B-V,0)$. This will imply the closedness of the form \eqref{quad-form}. We start with an auxiliary Lemma on heat kernels of certain Schr\"odinger operators with positive electric potential. Let $0\leq \rho \leq 1,\, \rho\neq 0$ be a radial function from $C^1(\R^2)$ with support in $\B_1$. Introduce a family of potential functions $U_\beta$ given as follows: 
\begin{equation} \label{ubeta}
U_\beta(x)=U_\beta(|x|)= \left\{
\begin{array}{l@{\quad}cr}
\beta^2 & \text{if } & |x| \leq 1  \\
\beta^2\, |x|^{-2} & \text{if}  & |x| >1 
\end{array}
\right. \quad \beta >0, \qquad 
U_0 (x) = U_0(|x|) = \rho(|x|).
\end{equation}
Next we define Schr\"odinger operators
$$
\A_\beta = -\Delta + U_\beta \quad \mbox{in\, } L^2(\R^2).
$$
In view of the standard Beurling-Deny criteria, the operators $\A_\beta$ generate contraction semigroups $e^{-t \A_\beta}$ on $L^2(\R^2)$ with almost everywhere positive integral kernels $e^{-t \A_\beta}(x,y)=:k_\beta(t,x,y)$. 

\begin{lemma} \label{abeta}
For almost every $x\in\R^2$ and all $t>0$ we have
\begin{equation} \label{hk-beta}
k_\beta(t,x,x)= e^{-t \A_\beta}(x,x)  \leq C\, \min\big\{\, t^{-1},\, (1+|x|)^{2\beta}\, \, t^{-1-\beta}\big\} \qquad \beta >0,
\end{equation}
and 
\begin{equation} \label{hk-0}
k_0(t,x,x)=e^{-t \A_0}(x,x)  \leq C\, 
\left\{
\begin{array}{l@{\quad}cr} 
t^{-1} & \text{if } & t \leq e  \\
\min\big\{\, t^{-1},\, (1+|\log |x||)^{2}\, \, t^{-1}\, (\log t)^{-2} \big\}  & \text{if}  & t >e 
\end{array}
\right. 
\end{equation}
for some constant $C$.
\end{lemma}

\begin{proof}
The spectrum of $\A_\beta$ coincides, for all $\beta\geq 0$, with the positive half-line $[0,\infty)$. Hence by the Allegretto-Piepenbrink theorem, see e.g. \cite{mp}, there exists a
positive solution $u_\beta$ to the equation $\A_\beta\, u_\beta=0$. Since the
potential $U_\beta$ is H\"older continuous, the elliptic regularity ensures that
$u_\beta\in C^2(\R^2)$. The radial function $h_\beta$ given by
$$
h_\beta(|x|) = \int_0^{2\pi} u_\beta (|x|,\theta)\, d\theta,
$$
then also satisfies $\A_\beta\, h_\beta=0$. Thus the weighted Laplace operator
\begin{equation}
-\Delta_\beta = h_\beta^{-1}\, \A_\beta\, h_\beta \qquad \mbox{in\,  \, \, } L^2(\R^2,\, h_\beta^2\, dx),
\end{equation}
generated by the quadratic form 
$$
\int_{\R^2} |\nabla u|^2\, h_\beta^2(x)\, dx, \qquad u\in H^1(\R^2,\, h_\beta^2\, dx),
$$
is unitarily equivalent to $\A_\beta$ and its heat kernel satisfies 
\begin{equation} \label{kernels-w}
e^{-t \A_\beta}(x,y) = h_\beta(x)\, h_\beta(y)\, e^{t \Delta_\beta}(x,y), \quad x,y \in \R^2.
\end{equation}
Now denote $r=|x|$ and observe that 
$$
(r\, h_\beta'(r))' = h_\beta(r)\, r\, U_\beta(r), 
$$
which implies that $h_\beta$ is increasing and that for $r>1$ it holds 
\begin{align}
h_\beta(r)  & = a_1\, r^\beta +b_1\, r^{-\beta}, \qquad \beta >0
\label{nonzero}\\
h_0(r)  & = a_2 +b_2\, |\log r|,  \qquad \, \,    \beta = 0.
\label{zero}
\end{align}
Since $h_\beta$ is positive and increasing it follows that $a_1>0,\, b_2>0$.  Thus for any $\beta\geq 0$ there exists  a constant $M_\beta$ such that 
\begin{equation} \label{help}
h_\beta(2r)\,  \leq\, M_\beta\,  h_\beta(r), \qquad \forall\, r\in \R_+.
\end{equation}
Let $V_\beta(x, s)$ denote the volume of the ball of radius $s$ centered in $x$ in the measure  $h_\beta^2\, dx$.  In view of \eqref{nonzero} and \eqref{zero} it easily follows that the manifold 
$(\R^2,\, h_\beta^2\, dx)$ satisfies the volume doubling property; i.e. there exists a constant $c$ such tat for any $s$ it holds
$$
V_\beta(x, 2 s)\, \leq\, c\, V_\beta(x, s).
$$
Equation \eqref{help} and the theorems \cite[Thm.5.7]{gs05} and \cite[Thm.2.7]{gs05} thus 
imply that the manifold $(\R^2,\, h_\beta^2\, dx)$ satisfies the Li-Yau estimate for its heat kernel:
\begin{equation} \label{li-yau}
e^{t \Delta_\beta}(x,y) \, \simeq \,  \frac{C\, e^{-c\, \frac{|x-y|^2}{t}}, }{\sqrt{V_\beta(x,\sqrt{t})}\, \, \sqrt{V_\beta(x,\sqrt{t})}}\, ,
\end{equation}
where $c$ and $C$ are positive constants. However, by \eqref{nonzero} and  \eqref{zero} we have
$$
V_\beta(x,\sqrt{t}) \, \simeq \,  t\, h^2_\beta(|x|+\sqrt{t}).
$$
Hence
\begin{equation} \label{heat-kernel-2d}
e^{-t \A_\beta}(x,y) \, \simeq \, C\, \frac{h_\beta(|x|)\, h_\beta(|y|)}{t\,
h_\beta(|x|+\sqrt{t})\, h_\beta(|y|+\sqrt{t})}\, \, e^{-c\, \frac{|x-y|^2}{t}}.
\end{equation}
Since  $h_\beta$ is increasing, this together with \eqref{kernels-w} and the estimate
$$
e^{-t \A_\beta}(x,y) \leq e^{t \Delta}(x,y) = \frac{1}{4\pi t}\, e^{-\frac{|x-y|^2}{4t}}  \qquad a.e.\  \ x,y \in\R^2,
$$
which follows by the Trotter product formula, imply equations \eqref{hk-beta} and \eqref{hk-0}. 
\end{proof}

%%%%%%%%%%%%%%%%%%%%%%%%%%%%%%%%%%%%%%%%%%%%%%%%%

\subsection{Proof of Theorem \ref{clr-mag-1}} Let $\chi_1$ be the characteristic function of $\B_1$. From \cite{timo} we know that the Hardy type inequality
\begin{equation} \label{hardy-t}
H_B \, \geq \, \gamma\, \chi_1
\end{equation}
holds, for some constant $\gamma>0$, in the sense of quadratic forms on $C_0^\infty(\R^2)$. 

\begin{proof}[Proof of Theorem \ref{clr-mag-1}] Let $a>0$. Inequality \eqref{hardy-t} and the variational principle imply that for any $\eps \in (0,1)$ we have  
\begin{align}
N(H_B-V,0) & \leq N(H_B-V,0) \leq N(\eps\, H_B+(1-\eps)\, c\, \chi_1- V,0) 
\nonumber \\ 
& \leq N\Big(H_B+\frac{c_1 (1-\eps)}{\eps} \, \chi_1-\eps^{-1}\,  V,0\Big), \label{basic-1}
\end{align}
where we have used the fact that multiplying an operator by a positive constant does not change the number of its negative eigenvalues. Next we chose $\eps$ such that 
$$
\frac{c_1 (1-\eps)}{\eps} \, \chi_1 \geq U_0,
$$
which is possible due to the hypotheses on $U_0$, so that 
$$
N(H_B-V,0)  \leq N(H_B+U_0 -  \eps^{-1}\,  V,0).
$$
For each $\beta \geq 0$ the operator $H_B+U_\beta$ generates a contractive semigroup $e^{-s(H_B+U_\beta)}$ in $L^2(\R^2)$. Let 
$$
K_\beta(s,x,y) := e^{-s(H_B+U_\beta)}(x,y)  \qquad x,y\in \R^2.
$$
be its integral kernel. By the the diamagnetic inequality, see e.g. \cite{sim2,hs}, we have
\begin{equation} \label{diamag}
\big| K_\beta(s,x,y) \big| \, \leq \, k_\beta(s,x,y), \qquad \beta\geq 0, \quad a.e.\ \ x,y\in\R^2, \quad s>0,
\end{equation}
This allows us to use a generalisation of the Lieb's inequality \cite{lieb}, see  \cite[Thm.2.5]{rs} or \cite{ahs,fls}, and therefore to obtain the upper bound
\begin{align} 
N(H_B+U_0- \eps^{-1}\,   V,0) & \leq\, C_\eps \int_0^\infty\, \frac 1t\, \int_{\R^2} k_0(t,x,x)\, (t\, V(x)-1)_+\, dx\, dt.  \nonumber\\
&  \leq C_\eps\int_{\R^2} \int_{1/V(x)}^\infty\,  k_0(t,x,x) \, V(x)\, dt \, dx. \label{lieb-2}
\end{align}
Next we set $t_0(x) = e +\frac{1}{V(x)}$ and perform the integration w.r.t. $t$ using the estimates
$$
k_0(t,x,x) \leq \frac{c}{t} ,\qquad 0< t < t_0(x), \qquad k_0(t,x,x) \leq \frac{c\, (1+|\log |x||)^{1+a}}{t \, (\log t)^{1+a}},\qquad t_0(x) \leq t, 
$$
which follow easily from \eqref{hk-0}. This gives inequality \eqref{eq-clr-1}. Moreover, the operator $H_B-V$ has only finitely many eigenvalues which shows that the quadratic form \eqref{quad-form} is bounded from below and therefore closable. 
\end{proof}

%%%%%%%%%%%%%%%%%%%%%%%%%%%%%%%%%%%%%%%%%%%%%%%%%

\subsection{Proof of Theorem \ref{clr-mag-2}} The arguments follow closely the proof of Theorem \ref{clr-mag-1}. In view of assumption \ref{ass-field} and \cite{lw} we have 
\begin{equation} \label{hardy-at}
H_B \, \geq \, c_B\, U_1
\end{equation}
in the sense of quadratic forms on $C_0^\infty(\R^2)$, where $c_B$ is a positive constant, see also \cite{bls}. 

\begin{proof}[Proof of Theorem \ref{clr-mag-2}] Fix $a>0$ and chose $\eps>0$ such that 
$$
a^2 = \frac{(1-\eps)\, c_B}{\eps}.
$$
Mimicking the argument used in \eqref{basic-1} and taking into account inequality \eqref{hardy-at} we get 
\begin{align}
N(H_B-V,0) & \leq  N(H_B+U_a-\eps^{-1}\,  V,0). \label{basic-2}
\end{align}
In the same way as in the proof of Theorem \ref{clr-mag-1} we arrive at 
$$
N(H_B+U_a-\eps^{-1}\,  V,0) \leq C_a \int_{\R^2} \int_{1/V(x)}^\infty\,  k_a(t,x,x) \, V(x)\, dt \, dx.
$$
Inequality \eqref{eq-clr-2} then follows from estimate \eqref{hk-beta}. 
\end{proof}

%%%%%%%%%%%%%%%%%%%%%%%%%%%%%%%%%%%%%%%%%%%%%%%%%

\subsection{Proof of Corollary \ref{cor-sobolev}} If $|A|$ is bounded, then the closure of $C_0^\infty(\R^2)$ with respect to the norm $\| (i\nabla +A) u\|_2^2+ \|u\|_2^2$ coincides with the Sobolev space $H^1(\R^2)$. Hence it suffices to prove \eqref{eq:sobolev} for $u\in C_0^\infty(\R^2)$. To this end we follow the approach of \cite{fls}. 

\begin{proof}[Proof of Corollary \ref{cor-sobolev}]
Let $u\in C_0^\infty(\R^2)$ and assume that $2 <q <\infty$. Let 
\begin{equation} \label{choice}
V(x) = \eta\, \Big ( C(B, 2/(q-2)) \int_{\R^2} |u(x)|^q\, (1+|x|)^{-2}\, dx\Big)^{\frac{2-q}{q}}\  |u(x)|^{q-2}\, (1+|x|)^{-2},
\end{equation}
where $0< \eta <1$ and $C\big(B,2/(q-2))$ is the constant in inequality \eqref{eq-clr-2}. It follows from  \eqref{eq-clr-2} that $N(H_B-V,0) =0$. Hence 
$$
\int_{\R^2}  |(i \nabla + A)\, u|^2\, dx \, \geq \, \int_{\R^2}  V(x)\, |u(x)|^2\, dx, 
$$
which implies \eqref{eq:sobolev}. If $q=2$, then the statement is equivalent to the Hardy inequality  \eqref{hardy-at}.
\end{proof}

%%%%%%%%%%%%%%%%%%%%%%%%%%%%%%%%%%%%%%%%%%%%%%%%%
\section{\bf Hardy inequalities} 

\noindent In this section we prove some Hardy type inequalities for the operator $H_B$. These inequalities   will be used in the proofs of Theorems \ref{clr-radial} and \ref{clr-radial-integer}.  

\begin{lemma} \label{hardy-log-lem}
Assume that $A\in L^2_{loc}(\R^2)$ generates a non-zero magnetic field. Then there exists a constant $C(A)>0$ such that 
\begin{equation} \label{hardy-log}
\int_{\R^2} |(\nabla +i A)\, u(x) |^2\, dx \, \geq \, C(A) \int_{\R^2} \frac{|u(x)|^2}{1+|x|^2\, \log ^2|x|}\, dx, \qquad \forall\, u\in C_0^\infty(\R^2).
\end{equation}
\end{lemma}

\begin{proof}
Let $u\in C_0^\infty(\R^2)$. By \cite[Thm.2.1]{timo} we have 
$$
\int_{\R^2} |(\nabla +i A)\, u(x) |^2\, dx \, \geq \, c_0 \int_{|x| \leq 3} |u(x)|^2\, \, dx, 
$$
for some $0< c_0 <1$. By Kato's inequality 
\begin{equation} \label{kato}
\|\nabla |u|\|^2 \leq  \| (\nabla +i A)\, u\|^2, \qquad u\in C_0^\infty(\R^2),
\end{equation}
see \cite{hsu,sim2},  it thus suffices to show that 
\begin{equation} \label{enough-h}
\int_0^\infty |f'(r) |^2\, r\, dr + c_0 \int_0^3 |f(r)|^2\, r\, dr \geq \, C \int_3^\infty \frac{|f(r)|^2}{r\, (\log r)^2}\,  dr 
\end{equation}
holds for all $f\in C_0^\infty(\R_+)$ and some constant $C>0$. Define the function $\phi$ by
$$
\phi(r) = \left\{
\begin{array}{l@{\quad}cr}
c_0 & \text{if \,} & 0< r \leq 1\\
c_0 (2-r)  & \text{if \,} & 1 < r \leq 2
\end{array}
\right.  ,
\qquad 
\phi(r) = \left\{
\begin{array}{l@{\quad}cr}
c_0(r-2) &  \, \text{if }  & 2 < r \leq 3 \\
c_0       & \,  \text{if }  & 3 < r  
\end{array}
\right. .
$$
A simple integration by parts then shows that 
$$
\int_0^\infty |(\phi f)'(r) |^2\, r\, dr +c_0(1-c_0) \int_0^3 |f(r)|^2\, r\, dr \leq \int_0^\infty |f'(r) |^2\, r\, dr + c_0 \int_0^3 |f(r)|^2\, r\, dr.
$$
On the other hand, since $\phi(2)=0$, integrating by parts again we obtain
$$
\int_2^\infty  \left( (\phi f)'(r) -\frac{(\phi f)(r)}{2r \log r} \right)^2 r\, dr = \int_2^\infty |(\phi f)'(r) |^2\, r\, dr - \int_2^\infty \frac{|(\phi f)(r)|^2}{4 r\, (\log r)^2}\, dr.
$$
Putting together the last two equations proves \eqref{enough-h} and hence \eqref{hardy-log}. 
\end{proof}

Hardy inequality \eqref{hardy-log} will have a crucial role in the proof of Theorem \ref{clr-radial-integer}.
Note that the weight function $(1+|x|^2\, \log ^2|x|)^{-1}$ on its right hand side belongs to $L^1(\R^2)$, cf. Lemma \ref{hardy-general} in section \ref{sec:hardy-integer}. On the other hand, Lemma \ref{no-log} below shows that on the orthogonal complement of the subspace of functions $u(x)=u(|x|)$ the logarithmic factor in \eqref{hardy-log} can be removed if the magnetic field is radial.  More general results in this direction concerning non-magnetic Hardy inequalities were obtained in \cite{sol2} .

\begin{lemma} \label{no-log}
Let the magnetic field satisfy conditions of Theorem \ref{clr-radial}. Then there exists a constant $\varkappa>0$ such that for any $u\in C_0^\infty(\R^2)$ we have 
\begin{equation} \label{hardy-complement}
 \int_0^{2\pi}\! u(r,\theta)\, d\theta=0 \quad \forall\, r>0 \quad \Rightarrow \quad \int_{\R^2} |(\nabla +i A)\, u(x) |^2\, dx \, \geq \, \varkappa \int_{\R^2} \frac{|u(x)|^2}{|x|^2}\, dx.
\end{equation}
\end{lemma}

\begin{proof} Let $u\in C_0^\infty(\R^2)$ satisfy the hypotheses in \eqref{hardy-complement}. 
Then we can decompose $u$ into the Fourier series  
$$
u(r,\theta) = \sum_{m\neq 0} u_m(r)\, \frac{e^{i m\theta}}{\sqrt{2\pi}}\, , \qquad u_m(r) = \frac{1}{\sqrt{2\pi}}\, \, \big(u(r,\cdot)\, ,\, e^{i m\theta}\big)_{L^2(0,2\pi)}. 
$$
For radial magnetic fields we have 
\begin{equation} \label{fourier}
\int_{\R^2} |(\nabla +i A)\, u|^2 =   \sum_{m\neq 0}  \int_0^\infty \Big( |u_m'(r)|^2 +\frac{(\Phi(r)+m)^2}{r^2}\, |u_m(r)|^2\Big)\, r\, dr,
\end{equation}
see equation \eqref{qm} in section \ref{sect-proofs-radial}. 
Since $\Phi(r)$ is bounded, there exist $c>0$ and $M_0\in\N$  such that 
\begin{equation} \label{aux-2}
(\Phi(r)+m)^2 \geq c >0 \qquad \forall\ r>0, \quad \forall\ m: \, |m| \geq M_0.
\end{equation}
On the other hand, in view of the fact that $\Phi(r)\to 0$ as $r\to 0$ and $\Phi\not\in\Z$, for any $m\neq 0$ we can find $0<r_m <R_m$ and a constant $c_m>0$ such that
$$
(\Phi(r)+m)^2 \geq c_m \quad \text{on\, } (0,r_m) \cup (R_m, \infty).  
$$
By "extending" the Hardy weight onto the interval $(r_m,  R_m)$ in the same way as it was done in Lemma \ref{hardy-log-lem} above, we then find out that 
$$
\forall\, m\neq 0,\, |m| < M_0 \ \exists \ \tilde c_m>0\, : \ 
\int_0^\infty \Big( |u_m'(r)|^2 +\frac{(\Phi(r)+m)^2}{r^2}\, |u_m(r)|^2\Big)\, r\, dr \, \geq \, \tilde c_m
 \int_0^\infty  \frac{|u_m(r)|^2}{r}\, dr.
 $$
Hence by \eqref{fourier}, \eqref{aux-2} and the Parseval's identity there exits a $\varkappa>0$ such that
\begin{equation*} 
\int_{\R^2} |(\nabla +i A)\, u|^2 \geq \, \varkappa  \int_{\R^2}  \frac{|u(x)|^2}{|x|^2}\, dx
\end{equation*} 
\end{proof}

\noindent If the total flux $\Phi$ is an integer, then we have  

\begin{lemma} \label{no-log-integer}
Let the magnetic field satisfy the conditions of Theorem \ref{clr-radial-integer}. Then there exists a constant $\varkappa'>0$ such that for any $u\in C_0^\infty(\R^2)$ the following holds: If  
\begin{equation} 
 \int_0^{2\pi}\! u(r,\theta)\, d\theta= \int_0^{2\pi}\! e^{-i \theta\, \Phi}\, u(r,\theta)\, d\theta  = 0 \qquad \forall\, r>0,
\end{equation}
then
\begin{equation} 
\int_{\R^2} |(\nabla +i A)\, u(x) |^2\, dx \, \geq \, \varkappa' \int_{\R^2} \frac{|u(x)|^2}{|x|^2}\, dx.
\end{equation}
\end{lemma}

\begin{proof}
This is a straightforward analogue of the proof of Lemma \ref{no-log}. 
\end{proof}

%%%%%%%%%%%%%%%%%%%%%%%%%%%%%%%%%%%%%%%%%%%%%%%%%

\section{\bf Proofs of the main results: radial fields}
\label{sect-proofs-radial}

For radial magnetic fields we introduce the  corresponding vector potential
$A$ in polar coordinates $(r,\theta)$ as follows:
$$
A(r,\theta) = a(r)\, (-\sin\theta,\, \cos\theta), \quad a(r) = \frac
1r  \int_0^r\! B(t)\, t\, dt = \frac 1r\, \Phi(r).
$$
Then ${\rm curl\, }A =B$. Since $A$ is bounded, in view of Assumption \ref{ass-field-radial}, the Hamiltonian $H_B$ is associated with the closed quadratic form
\begin{equation} \label{q-form}
\int_0^\infty\! \int_0^{2\pi} \left(|\pd_r u|^2+ r^{-2}
|i\, \pd_\theta u+\Phi(r) u|^2\right) r\, dr d\theta, \quad u\in
H^1(\R_+\times(0,2\pi)).
\end{equation}
By expanding a given function $u\in L^2(\R_+\times (0,2\pi))$ into a
Fourier series with respect to the orthonormal basis $\{ (2\pi)^{-1/2}\, e^{i
  m\theta}\}_{m\in\Z}$ of $L^2(0,2\pi)$, we obtain
a direct sum decomposition
\begin{equation}
L^2(\R^2) = \sum_{m\in\Z} \oplus \, \mathcal{L}_m,
\end{equation}
where $\mathcal{L}_m= \left\{g\in L^2(\R^2)\, : \,  g(x)= f(r)\,
e^{i m\theta} \, a.e., \, \int_0^\infty |f(r)|^2\,  r\,
dr<\infty\right\}$. Since the magnetic field $B$ is radial, the
operator $H_B$ can be decomposed accordingly to the direct sum
\begin{equation} \label{sum-gen}
H_B =   \sum_{m\in\Z}  \oplus \left( h_m \otimes\mbox{id}\right)
\Pi_m,
\end{equation}
where $h_m$ are operators generated by the closures, in $L^2(\R_+, r
dr)$, of the quadratic forms
\begin{equation} \label{qm}
\int_0^\infty\, \left(|f'|^2+\frac{(\Phi(r)+m)^2}{r^2}\,
|f|^2\right)\, r\, dr
\end{equation}
defined initially on $C_0^\infty(\R_+)$, and $\Pi_m:L^2(\R^2)\to
\mathcal{L}_m$ is the projector acting as
\begin{equation} \label{pim}
(\Pi_m\,  u)(r,\theta) = \frac{1}{2\pi}\, \int_0^{2\pi}
e^{im(\theta-\theta')}\, u(r,\theta')\, d\theta'.
\end{equation}
Obviously, the operator $H_0=-\Delta$ admits a similar decomposition:
\begin{equation}
-\Delta =   \sum_{m\in\Z}  \oplus \left( P_m \otimes\mbox{id}\right) \Pi_m,
\end{equation}
where $P_m$ are operators generated by the closures, in $L^2(\R_+, r
dr)$, of the quadratic forms
$$
\int_0^\infty\, \left(|f'|^2+\frac{m^2}{r^2}\, |f|^2\right)\, r\, dr, \qquad f\in C_0^\infty(\R_+).
$$

%%%%%%%%%%%%%%%%%%%%%%%%%%%%%%%%%%%%%%%%%%%%%%%%%

\subsection{Proof of Theorem \ref{clr-radial}}  
\label{sect-radial1}
We prove the upper bound \eqref{clr-eq-radial}  for continuous and compactly supported $V$. The general case then follows by approximating $V$ by a sequence of continuous compactly supported functions and using a standard limiting argument in inequality \eqref{clr-eq-radial}. Let $\Pi_0$ be given by \eqref{pim} and let   
$$
 Q\, u = u-\Pi_0 \, u\, ,\quad u\in L^2(\R^2)\, .
$$
Since $\Pi_0$ and $Q$ commute with $H_B$, the variational principle and the inequality 
$$
|(u, (\Pi_0 V Q + Q V \Pi_0) u) | \, \leq (u, Q VQ\, u) + (u, \Pi_0 V \Pi_0\, u) \qquad \forall\ u \in C_0^\infty(\R^2)  
$$
imply that the estimate
\begin{equation} \label{variation}
H_B-V \geq \Pi_0\, (H_B- 2\, V)\, \Pi_0 +Q\, (H_B- 2\, V)\, Q
\end{equation}
holds true in the sense of quadratic forms on $C_0^\infty(\R^2)$. Hence 
\begin{equation} \label{upperestimate}
N(H_B-V, 0) \leq \, N( \Pi_0\, (H_B - 2\, V)\, \Pi_0, 0) + N( Q\, (H_B - 2\, V)\, Q, 0)
\end{equation}
Set 
\begin{equation} \label{average}
\hat V (r) = \frac{1}{2\pi}\, \int_0^{2\pi} V(r,\theta)\, d\theta\, .
\end{equation}

\noindent Let us denote by $P_0^{a,b}$ the restriction of the operator $P_0$ on $L^2((a,b), r dr)$ with Neumann boundary conditions at the end points $a$ and $b$. 

\begin{lemma}\label{aux-1dim}
Let $0\leq a < b \leq \infty$. Assume that $W\geq 0$ is continuous and compactly supported. Then there exists a constant $L_0$, independent of $a$ and $b$, such that for any $\delta>0$  we have  
\begin{equation} \label{eq:1dim}
N\big(P_0^{a,b} +\frac{\delta^2}{r^2} -W(r) ,0\big)_{L^2((a,b), r dr)} \, \leq \, \frac{L_0}{\delta} \int_a^bW(r)\, r\, dr.
\end{equation}
\end{lemma}

\begin{proof}
Consider the mapping from $\U :L^2((a,b), r dr) \mapsto L^2(a,b)$ defined by $(\U f)(r) = r^{1/2} f(r)$. A direct calculation shows that the operator  
\begin{equation} \label{operator-t}
T_\delta^{a,b} := \U \, \Big(P_0^{a,b} +\frac{\delta^2}{r^2}\, \Big) \, \U ^{-1} \quad \text{in \ } L^2(a,b)
\end{equation}
acts on its domain according to 
\begin{align*}
( T_\delta^{a,b}\, u)(r) & = -u''(r) +\frac{\delta^2-\frac 14}{r^2}\, \, u(r), \qquad  u'(a) = \frac{u(a)}{2a}, \quad u'(a) = \frac{u(b)}{2b}, \quad 0 < a <b < \infty,
\end{align*}
where the boundary conditions take the form $u(a)=0$ if $a =0$ and $u\in L^2(a,\infty)$ if $b=\infty$. Let $G_\delta^{a,b}(r,r',\kappa)$ be the integral kernel of the resolvent of $T_\delta^{a,b}$ at the point $\kappa^2$, i.e. 
$$
G_\delta^{a,b}(r,r',\kappa) = \big(T_\delta^{a,b} + \kappa^2\big)^{-1}(r,r').
$$
From the Sturm-Liouville theory of ordinary differential operators we calculate 
$$
G_\delta^{a,b}(r,r,\kappa) =
\frac{r}{\omega_{\delta}(a) +\omega_{\delta}(b)}\, (I_{\delta}(r \kappa) +\omega_{\delta}(a)\, K_{\delta}(r \kappa))(I_{\delta}(r \kappa) +\omega_{\delta}(b)\, K_{\delta}(r\kappa)),
$$
where $I_\delta$ and $K_\delta$ are the modified Bessel functions, and
\begin{equation}
\omega_\delta(r) = -\frac{I'_\delta(r\, \kappa)}{K'_\delta(r\, \kappa)}. 
\end{equation}
From \cite[Sect.9.6]{as} we then deduce that 
$$
\lim_{\kappa\to 0} G_\delta^{a,b}(r,r,\kappa) = \frac{2 r}{\delta}\, \Big(1+ \frac{2^{2-2\delta}\,  (a\, b)^{2\delta}}{(a^{2\delta}+b^{2\delta})\, r^{2\delta}} \Big) \, \leq \, \frac{c\, r}{\delta},
$$
with a constant $c$  independent of $a,b$ and $r$. The Birman-Schwinger principle thus gives 
$$
N\big(T_\delta^{a,b} -W(r) ,0\big)_{L^2(a,b)}\, \leq\, \lim_{\kappa\to 0} \int_a^b G_\delta^{a,b}(r,r,\kappa)\, W(r)\, dr \, \leq \, \frac{c}{\delta} \int_a^bW(r)\, r\, dr.
$$
Since $\U$ is unitary, this proves the statement. 
\end{proof}

\begin{lemma} \label{orthog}
Let $V\in L^1(\R_+, L^\infty(\Sph^1))$. Then for any $\eps>0$ there exists a $C_\eps$ such that 
\begin{equation} \label{eq:orthog}
N\big(H_B +\frac{\eps}{|x|^2} - V, 0) \, \leq \, C_\eps\, \|V\|_{L^1(\R_+, L^\infty(\Sph^1))}. 
\end{equation}
\end{lemma}

\begin{proof}
By density, it suffices to prove the estimate for continuous and compactly supported $V$. By \eqref{sum-gen} we have 
\begin{equation} \label{N-sum}
N\big(H_B +\frac{\eps}{|x|^2} - V, 0) \leq \sum_{m\in\Z} N( h_m +\frac{\eps}{|x|^2} - \tilde V, 0)_{L^2(\R_+, r dr)}
\end{equation}
We recall the result of Laptev \cite{la}:
\begin{equation}\label{ln-eq}
N\big( -\Delta+\frac{\eps}{|x|^2} - \tilde V, 0\big)  = \sum_{m\in\Z} N( P_0 +\frac{m^2+\eps}{r^2} - \tilde V, 0)_{L^2(\R_+, r dr)} \, \leq \, c(\eps)\, \|V\|_{L^1(\R_+, L^\infty(\Sph^1))}. 
\end{equation}
Since $\Phi(r)$ is bounded, there exists $n_0 \in \N$ such that 
$$
(m+\Phi(r))^2 \geq \frac{m^2}{2} \qquad \forall\ r >0 \quad \text{and} \quad  \forall\ m\in\Z\ : \  |m| > n_0.
$$
Hence from \eqref{ln-eq} it easily follows that 
\begin{align*} 
\sum_{|m| >n_0} N( h_m +\frac{\eps}{r^2} - \tilde V, 0)_{L^2(\R_+, r dr)} & \leq \sum_{|m| >n_0} N( P_0 +\frac{m^2+ \eps}{r^2} - 2 \tilde V, 0)_{L^2(\R_+, r dr)} \\
& \leq N\big( -\Delta+\frac{\eps}{|x|^2} - 2\, \tilde V, 0\big) \, \leq 2\, c(\eps) \, \|V\|_{L^1(\R_+, L^\infty(\Sph^1))}.
\end{align*}
On the other hand, by Lemma \ref{aux-1dim}  we have for any $m\in\Z$
$$
N( h_m +\frac{\eps}{r^2} - \tilde V, 0)_{L^2(\R_+, r dr)} \, \leq \, N( P_0 +\frac{\eps}{r^2} -  \tilde V, 0)_{L^2(\R_+, r dr)} \, \leq \tilde C_\eps\, \|V\|_{L^1(\R_+, L^\infty(\Sph^1))}.
$$
In view of \eqref{N-sum}, this completes the proof.
\end{proof}

\begin{lemma} \label{m=0}
Let $B$ satisfy the hypotheses of Theorem \ref{clr-radial}. Assume that $V$ is continuous and compactly supported. Then there exists a constant $c_0$ such that 
\begin{equation} \label{clr-1dim}
N( \Pi_0\, (H_B - \, V)\, \Pi_0, 0) \, \leq \, c_0\, \big( \, \|V \|_{L^1(\R^2 )} + \|V\log |x| \|_{L^1(\B_1)} \big).
\end{equation}
\end{lemma}

\begin{proof}
In view of the Hardy inequality \eqref{hardy-at} it suffices to prove 
\begin{equation} \label{enough}
N( \Pi_0\, (H_B + U_1- \, V)\, \Pi_0, 0) \, \leq \, c\, \big( \, \|V \|_{L^1(\R^2 )} + \|V\log |x| \|_{L^1(\B_1)} \big),
\end{equation}
where $U_1$ is given by \eqref{ubeta}. Note that
\begin{equation} \label{halfline}
N( \Pi_0\, (H_B + U_1- \, V)\, \Pi_0, 0)  = N(h_0+U_1 - \hat V, 0)_{L^2(\R_+, r dr)}  .
\end{equation}
We impose additional Neumann boundary condition at the point $r=1$. By the variational principle 
$$
N(h_0+U_1 - \hat V, 0)_{L^2(\R_+, r dr)} \, \leq \, N(P_0^{0,1} +1- \hat V, 0)_{L^2((0,1), r dr)} \, + N(P_0^{1,\infty} +\frac{1}{r^2} - \hat V, 0)_{L^2((1,\infty), r dr)},
$$
Moreover, Lemma \ref{aux-1dim} implies that for some $c$ it holds
\begin{equation} \label{est-T2}
N(P_0^{1,\infty} +\frac{1}{r^2} - \hat V, 0)_{L^2((r_0,\infty), r dr)}\, \leq \, c\, \int_{1}^\infty \hat V(r)\, r\, dr.
\end{equation}
As for the operator $P_0^{0,1} +1$ in $L^2((0,1), r dr)$, we note that $\inf \sigma (P_0^{0,1} +1) = 1$. 
Hence 
\begin{align}  \label{aux-3}   
& N(P_0^{0,1} +1 -\hat V, 0)_{L^2((0, 1), r dr)}  = N(P_0^{0,1} -\hat V, -1)_{L^2((0, 1), r dr)} 
= N (T_0 - \hat V, -1)_{L^2(0,1)}, 
\end{align}
where $T_0 = \U\, P_0^{0,1}\, \U^{-1}$ is the operator in $L^2(0,1)$ acting on its domain as
$$
(T_0\, u )(r) = -u''(r) -\frac{u(r)}{4 r^2}  \quad \text{with boundary conditions\, \, } u'(1)= \frac{u(1)}{2}, \quad u(0)=0.
$$
As above we calculate the diagonal element of the integral kernel of $(T_0+\kappa^2)^{-1}$:  
\begin{align*}
(T_0+\kappa^2)^{-1}(r,r)=: G_0(r,r, \kappa) = r\, I_0(r \kappa) \big(K_0(\kappa r) +\omega^{-1}_0(1)\, I_0(r \kappa) \big). 
\end{align*}
Using the properties of functions $I_0$ and $K_0$, see e.g. \cite[Sect.9.6]{as}, it is then easy to verify that
$$
G_0(r,r, 1)\, \leq \, c\ r\, (1 + |\log r|) \quad r\in (0,1). 
$$
The Birman-Schwinger principle and equation \eqref{aux-3} then yield
\begin{align} 
N(P_0^{0,1} +1 -\hat V, 0)_{L^2((0, 1), r dr)} & = N(T_0 - \hat V, -1)_{L^2(0,1)} \, \leq  \int_0^{1} G_0(r,r,1)\, \hat V(r)\, dr \nonumber \\
& \leq  c \int_0^{1} \hat V(r) (1+|\log r|)\, r \, dr.  \label{est-T1}
\end{align}
This in combination with \eqref{est-T2} implies  \eqref{enough} and therefore the statement of the Lemma. 
\end{proof}

\begin{proof}[Proof of Theorem \ref{clr-radial}] 
Lemma \ref{no-log}, inequality \eqref{eq:orthog}  and the variational principle yield 
$$
N( Q\, (H_B - V)\, Q, 0)\, \leq \, N\big( Q\, (H_B +\frac{\varkappa}{|x|^2} - 2\, V)\, Q, 0\big)\, \leq \, N\big(H_B +\frac{\varkappa}{|x|^2} - 2\, V, 0\big) \leq c\, \|V\|_{L^1(\R_+, L^\infty(\Sph^1))}.
$$
The proof is completed by using Lemma \ref{m=0}.   
\end{proof}

\begin{remark}
Similar estimates, in terms of logarithmic Lieb-Thirring inequalities, for the operator $-\Delta-V$ in dimension two were obtained  in \cite{kvw}. Upper bounds on $N(-\Delta-V,0)$ including logarithmic weights were studied in \cite{chkmw, sol,timo2}. 
\end{remark}

%%%%%%%%%%%%%%%%%%%%%%%%%%%%%%%%%%%%%%%%%%%%%%%%%

\subsection{Proof of Theorem \ref{clr-radial-integer}}  
By Lemma \ref{hardy-log-lem} it suffices to prove the upper bound \eqref{clr-eq-integer} for the operator 
$$
H_B +\frac{1}{1+|x|^2\, \log ^2|x|}\ - V.  
$$

\begin{lemma} \label{zeroflux}
Let $B$ satisfy hypotheses of Theorem \ref{clr-radial-integer} and suppose that $\Phi=0$. Assume that $V$ is continuous and compactly supported. Then there exists a constant $L_1$ such that 
\begin{equation} \label{clr-1dim-integer}
N\Big(h_0 +\frac{1}{1+r^2\, \log ^2 r}\, -\hat V , 0\Big)_{L^2(\R_+, r dr)} \, \leq \, L_1\, \big( \, \|V \|_{L^1(\R^2 )} + \|V\log |x| \|_{L^1(\R^2)} \big)
\end{equation}
\end{lemma}

\begin{proof}
We impose addition Neumann boundary condition at $r=2$. By Neumann bracketing we have 
\begin{align*}
N\big(h_0 +\frac{1}{1+r^2\, \log ^2 r}\, -\hat V , 0\big)_{L^2(\R_+, r dr)} & \leq N\big(P_0^{0,2} +\frac{1}{1+r^2\, \log ^2 r}\, -\hat V , 0\big)_{L^2((0,2), r dr)} \\
& \quad +N\big(P_0^{2,\infty} +\frac{1}{r^2\, \log ^2 r}\,  -2\, \hat V , 0\big)_{L^2((2,\infty), r dr)}.
\end{align*}
A straightforward modification of \eqref{est-T1} gives
\begin{align}
N\big(P_0^{0,2}  +\frac{1}{1+r^2\, \log ^2 r}\, -\hat V , 0\big)_{L^2((0,2), r dr)}  & \ \leq c\, \int_0^2 \hat V(r)\, (1+\chi_{(0,1)}(r)\, |\log r|)\, r\, dr. \label{02}
\end{align}
On the interval $(2,\infty)$ we impose additional Neumann boundary conditions at $\{r=n,\, n\in\N,\, n\geq 3\}$. Hence
\begin{equation} \label{upper-sum}
N\big(P_0^{2,\infty} +\frac{1}{r^2\, \log ^2 r}\, -\hat V , 0\big)_{L^2((2,\infty), r dr)} \, \leq\, \sum_{n=2}^\infty N\big(P_0^{n,n+1} +\frac{1}{r^2\, \log ^2 r}\, -\hat V , 0\big)_{L^2((n,n+1), r dr)}.
\end{equation}
In the notation of the roof of Lemma \ref{aux-1dim}, see equation \eqref{operator-t}, we then obtain
\begin{align}
N\big(P_0^{n,n+1} +\frac{1}{r^2\, \log ^2 r}\, -\hat V , 0\big)_{L^2((n,n+1), r dr)}  & \leq N\big(T^{n,n+1}_{\delta_n}  -\hat V , 0\big)_{L^2(n,n+1)}  \, , \label{sigma_n}
\end{align}
where 
$$
\delta^2_n = \frac{1}{\log^2 (n+1)} .
$$
Hence in view of \eqref{eq:1dim} we get
$$
N\big(T^{n,n+1}_{\delta_n}  -\hat V , 0\big)_{L^2(n,n+1)} \, \leq \, L_0\, \int_n^{n+1} \delta_n^{-1}\, \hat V(r)\, r\, dr \, \leq\, \tilde c\, \int_n^{n+1} \hat V(r)\, ( \log r)\, r\, dr. 
$$
This together with \eqref{02} and \eqref{upper-sum} proves the Lemma.
\end{proof}

\begin{lemma} \label{flux-m}
Let $B$ satisfy the hypotheses of Theorem \ref{clr-radial-integer} and suppose that $\Phi=-m\in\Z$. Assume that $V$ is continuous and compactly supported. Then there exist constants $k_1$ and $k_2$ 
such that 
\smallskip
\begin{align} \label{clr-1dim-integer-1}
N\big(h_m +\frac{1}{1+r^2\, \log ^2 r}\, -\hat V , 0\big)_{L^2(\R_+, r dr)} \, & \leq \, k_1\, \big( \, \|V \|_{L^1(\R^2 )} + \|V\log |x| \|_{L^1(\R^2)} \big) \\
& \nonumber \\
 \label{clr-1dim-integer-2}
N\big(h_0 -\hat V , 0\big)_{L^2(\R_+, r dr)} \, & \leq \, k_2\, \big( \, \|V \|_{L^1(B_1)} + \|V \log |x| \|_{L^1(\R^2)} \big).
\end{align}
\end{lemma}

\begin{proof}
Inequalities \eqref{clr-1dim-integer-1} and \eqref{clr-1dim-integer-2} follows from straightforward modifications of Lemmata \ref{zeroflux} and \ref{m=0} respectively. 
\end{proof}

\begin{proof}[Proof of Theorem \ref{clr-radial-integer}] Assume that $\Phi=-m\in\Z$.
By inequality \eqref{hardy-log} it follows that 
\begin{equation} \label{N-hardy}
N(H_B-V,\, 0) \, \leq\, N\big(H_B  +\frac{1}{1+|x|^2\, \log ^2|x|}\ - \frac{2}{\varkappa}\, V,\, 0\big). 
\end{equation}
Let  $\Q = \id - \Pi_0-\Pi_m$ be the projection on the orthogonal complement of $\mathcal{L}_0 \oplus \mathcal{L}_m$. Mimicking the arguments of section \ref{sect-radial1} we obtain
\begin{align}
& N\big(H_B  +\frac{1}{1+|x|^2\, \log ^2|x|}\ - V,\, 0\big)   \leq  N\big(\Q (H_B  - 3V)\Q,\, 0\big) + N\big(\Pi_0(H_B  - 3V)\Pi_0,\, 0\big)\nonumber  \\
&\qquad \qquad + N\big(\Pi_m (H_B  +\frac{1}{1+|x|^2\, \log ^2|x|}\ - 3V) \Pi_m,\, 0\big)  = N\big(\Q(H_B - 3 V)\Q,\, 0\big)\nonumber \\  
& \qquad \qquad + 
N\big(h_0 - 3 \hat V , 0\big)_{L^2(\R_+, r dr)}  
 +N\big(h_m +\frac{1}{1+r^2\, \log ^2 r}\, -3  \hat V , 0\big)_{L^2(\R_+, r dr)}. \label{3-terms}
\end{align}
As in the proof of Theorem \ref{clr-radial} we note that by Lemma \ref{no-log-integer} and inequality \eqref{eq:orthog} 
$$
N\big(\Q(H_B -  V)\Q,\, 0\big) \leq N\big(H_B +\frac{\varkappa'}{|x|^2} - 2\, V,\, 0 \big) \leq c\, \|V \|_{L^1(\R_+, L^\infty(\Sph^1))}.
$$
The statement of the Theorem then follows from Lemmata \ref{zeroflux}, \ref{flux-m} and inequalities \eqref{N-hardy}, \eqref{3-terms}. 
\end{proof}

 It should be pointed out that the difference between the estimates \eqref{clr-eq-radial} and \eqref{clr-eq-integer}, in other words between the presence of the terms $\|V \log |x| \|_{L^1(\B_1)}$ and $ \|V \log |x| \|_{L^1(\R^2)}$, is a direct consequence of the decay rate of the respective Hardy weights:
$$
(H_B\, u\, , u) \, \geq \, (\rho\, u\, , u) \quad \forall\, u\in C_0^\infty(\R^2), \qquad  
\rho(x) = c\,  \left\{
\begin{array}{l@{\quad}cr}
\big(1+|x|^2 \log ^2|x| \big) ^{-1} & \text{if \,} &  \Phi\in\Z \\
&& \\
\big(1+|x|^2\big)^{-1}   & \text{if \, } & \Phi\notin\Z 
\end{array}
\right.  .
$$ 

\begin{remark}
The logarithmic factor in the case $\Phi\in\Z$  is specific to $\R^2$. For example in a waveguide-type domain $\R\times (0,1)$ the Hardy weight decays at infinity as  $|x|^{-2}$ independently of the total flux, cf.~\cite{ek}.
\end{remark}

\smallskip

%%%%%%%%%%%%%%%%%%%%%%%%%%%%%%%%%%%%%%%%%%%%%%%%%
\subsection{Proof of Proposition \ref{examples}} \label{examples:sec}

\begin{proof}[Proof of Proposition \ref{examples}] 
By \cite[Sec.6]{bl} for $\sigma >1$ we have
\begin{equation} \label{bl:eq}
\lim_{\lambda\to\infty} \lambda^{-\sigma} N(P_0-\lambda\, W_\sigma,0)_{L^2(\R+, r dr)}= \lim_{\lambda\to\infty} \lambda^{-\sigma} N(P_0-\lambda\, V_\sigma,0)_{L^2(\R+, r dr)}= \frac{4^{\sigma-1}\, \Gamma\big(\sigma-\frac 12\big)}{\sqrt{\pi}\, \, \Gamma(\sigma)}.
\end{equation}
Since $V$ is radial, the operator $H_B-V$ admit a decomposition analogous to \eqref{sum-gen}. Hence 
\begin{equation} \label{aux-sum-V}
N(H_B-\lambda\, V,\, 0)  = \sum_{m\in\Z}\, N(h_m-\lambda\, V,\, 0)_{L^2(\R_+, r dr)} \geq N(h_0-\lambda\, V,0)_{L^2(\R_+, r dr)}. 
\end{equation}
From the hypotheses on $B$ and the H\"older inequality we obtain
$$
\frac{\Phi(r)^2}{r^2} = o\big(V_\sigma(r) \big), \qquad r\to 0.
$$ 
This and a standard Dirichlet-Neumann bracketing yield 
\begin{equation} \label{dn-1}
\lim_{\lambda\to\infty} \lambda^{-\sigma} N(h_0-\lambda\,  V_\sigma,0)_{L^2(\R_+, r dr)} =  \lim_{\lambda\to\infty} \lambda^{-\sigma} N(P_0-\lambda\, V_\sigma,0)_{L^2(\R_+, r dr)}.
\end{equation}
The variational principle together with \eqref{bl:eq} and \eqref{dn-1} then imply that 
\begin{align*}
\liminf_{\lambda\to\infty} \lambda^{-\sigma} N(h_0-\lambda\, V,0)_{L^2(\R_+, r dr)} \, & \geq \,    
\lim_{\lambda\to\infty} \lambda^{-\sigma} N(P_0-\lambda\, c\, V_\sigma,0)_{L^2(\R_+, r dr)} >0.  
\end{align*}
where $c>0$ is a suitable constant. In view of  \eqref{aux-sum-V} this proves the first statement of the Proposition. 

To prove the second statement  assume that $\Phi(r)=-k\in\Z$ for all $r$ large enough.  The same reasoning as above shows that 
\begin{align*}
\liminf_{\lambda\to\infty} \lambda^{-\sigma}  N(H_B-\lambda\, V,0) &  \geq \liminf_{\lambda\to\infty} \lambda^{-\sigma}  N(h_{k}-\lambda\, c\, W_\sigma,0)_{L^2(\R_+, r dr)} \\
& =  \lim_{\lambda\to\infty} \lambda^{-\sigma} N(P_0-\lambda\, c\, W_\sigma,0)_{L^2(\R_+, r dr)} >0.
\end{align*}
\end{proof}

\begin{remark} \label{rem-singular}
From the proof of Proposition \ref{examples} it is clear that the super-linear growth of $N(H_B-\lambda\, V_\sigma)$ appears as long as the magnetic field does not have a strong singularity at the origin.  More precisely, for \eqref{noweyl:local} to fail the term $\Phi^2(r)/r^2$ would have to dominate the singularity of $V_\sigma(r)$ as $r\to 0$. This is for example the case of the Aharonov-Bohm field, when $\Phi(r)$ is constant, see Remark \ref{rem:ab}.   
\end{remark}

%%%%%%%%%%%%%%%%%%%%%%%%%%%%%%%%%%%%%%%%%%%%%%%%

\section{\bf Decay rate of Hardy weights} \label{sec:hardy-integer}

We have mentioned that the non-linear growth of $N(H_B-\lambda\, V)$ in $\lambda$ for potentials with a local singularity cannot be removed if the magnetic field is sufficiently regular.  Next we will discuss the behavior of $N(H_B-\lambda\, V)$ for slowly decaying potentials and in particular the connection between the non-linear growth of $N(H_B-\lambda\, V),\, V\in\W_\sigma$ and the decay rate of the weight function $\rho$ in the Hardy inequality 
\begin{equation} \label{hardy-abst}
H_ B \, \geq \, \rho(x)>0.
\end{equation}
Proposition \ref{examples} suggests that in order to suppress the super-linear growth of $N(H_B-\lambda\, V), \, V\in \W_\sigma$, the magnetic field should generate a Hardy inequality with a positive weight function $\rho$ dominating all the potentials from $\W_\sigma$ at infinity. 
From the definition of $\W_\sigma$ it follows that such weight function must satisfy $\rho\notin L^1(\R^2)$. 
This is the case of magnetic fields with non-integer flux, when $\rho(x) \simeq |x|^{-2}$ at infinity, see inequality \eqref{hardy-at}. However, in the case of integer flux we have

\begin{lemma} \label{hardy-general}
Assume that  $A\in L^\infty(\R^2)$ generates a bounded radial magnetic field with compact support and such that $\Phi=k\in\Z$. Suppose that 
\begin{equation} \label{hardy-L1}
\int_{\R^2} |(\nabla +i A)\, u(x) |^2\, dx \, \geq \, \int_{\R^2} |u(x)|^2\, \rho(x)\,  dx, \qquad \forall\, u\in C_0^\infty(\R^2).
\end{equation}
holds for some $0 \leq \rho\in L^\infty(\R^2), \, \rho\not\equiv 0$. Then $\rho \in L^1(\R^2)$. 
\end{lemma}

\begin{proof}
Assume that \eqref{hardy-L1} holds for some $0\leq \rho \notin L^1(\R^2)$.  Then, by density \eqref{hardy-L1} holds for all $u\in H^1(\R^2)$. Consider the family of test functions $u_n \in H^1(\R^2)$ given by 
\begin{equation} \label{test-f}
u_n(r,\theta) = e^{- ik \theta}\, \min \big\{ \big(\log (r n)\big)_+ ,\,  1,\, \big(\log (e\, n/r)\big)_+ \big\}.
\end{equation}
A straightforward calculation shows that 
$$
\sup_{n\in\N}\, \int_{\R^2} |(\nabla +i A)\, u_n |^2 = 2\pi\, \sup_{n\in\N}\, \int_0^\infty \Big( |u_n'|^2 +\frac{(\Phi(r)-k)^2}{r^2}\, |u_n|^2\Big)\, r\, dr < \infty,
$$
while $\int_{\R^2} |u_n|^2\, \rho \to \infty$ as $n\to\infty$ since $u_n$ converges almost everywhere to $1$. This contradicts \eqref{hardy-L1}. 
\end{proof}

\begin{remark} \label{counter-ex}
The arguments of the above proof also show, using the same family of test functions \eqref{test-f}, that Sobolev inequality \eqref{eq:sobolev} fails in the absence of magnetic field, or even in the presence of a magnetic field which satisfies conditions of Lemma \ref{hardy-general}. 
\end{remark}

%\smallskip

%%%%%%%%%%%%%%%%%%%%%%%%%%%%%%%%%%%%%%%%%%%%%%%%%%%%%%%%%%%%%%%%%%%%%%%%%%%%%%%%%%%%

\section*{Acknowledgements}
The research was partially supported by the MIUR-PRINÕ08 grant for the project  ''Trasporto ottimo di massa, disuguaglianze geometriche e funzionali e applicazioni''.

%\newpage
%%%%%%%%%%%%%%%%%%%%%%%%%%%%%%%%%%%%%%%%%%%%%%%%%%%%%%%%%%%%%%%%%%%%%%%%%%%%%%%%%%

\bibliographystyle{amsalpha}

\end{document}